\documentclass[12pt]{amsart}
\usepackage{amsmath,amssymb,amsfonts,amsthm,amsopn,dsfont}
\usepackage{graphics}

\newcommand{\ft}{Fourier transform}

\newcommand{\modsp}{modulation space}

\newcommand{\fpc}{\Fur L^p(\R^d)_{\rm comp}}

\newtheorem{tm}{Theorem}[section]
\newtheorem{lemma}[tm]{Lemma}

\newtheorem{theorem}{Theorem}[section]
\newtheorem{corollary}[theorem]{Corollary}

\newtheorem{proposition}[theorem]{Proposition}

\newcommand{\beqa}{\begin{eqnarray*}}
\newcommand{\eeqa}{\end{eqnarray*}}

\DeclareMathOperator*{\supp}{supp}

\newcommand{\field}[1]{\mathbb{#1}}
\newcommand{\bR}{\field{R}}        
\newcommand{\bN}{\field{N}}        
\newcommand{\bZ}{\field{Z}}        
\def\la{\lambda}

 \def\cF{\mathcal{F}}              
 \def\cS{\mathcal{S}}

 \def\cM{\mathcal{M}}

 \def\cC{\mathcal{C}}

\def\a{\aleph}

\def\rd{\bR^d}

\def\rdd{{\bR^{2d}}}

\def\lrd{L^2(\rd)}

\def\zd{\bZ^d}

\def\intrd{\int_{\rd}}
\def\intrdd{\int_{\rdd}}

\def\R{\right)}

\def\<{\left<}
\def\>{\right>}

\def\inv{^{-1}}

\def\mv1{M_v^1}

\def\phas{(x,\o )}
\def\mn{(m,n)}
\def\mn'{(m',n')}

\def\o{\eta}
\def\a{\alpha}
\def\b{\beta}

\def\Z{\mathbb{Z}^{d}}

\def\N{\mathbb{N}}
\def\R{\mathbb{R}}
\def\Ren{\mathbb{R}^d}

\def\sch{\mathcal{S}}

\def\Fur{\mathcal{F}}

\def\f{\varphi}

\def\Sn2{S_{2}(L^{2}(\Ren))}
\def\S1{S_{1}(L^{2}(\Ren))}
\def\sig00{\sigma_{0,0}}

\def\la{\langle}
\def\ra{\rangle}

\begin{document}

\begin{abstract}
We study the action of
Fourier Integral Operators
(FIOs) of H{\"o}r\-man\-der's
type on $\fpc$, $1\leq
p\leq\infty$. We see, from
the Beurling-Helson theorem,
 that generally FIOs of order zero fail
to be bounded on these
spaces when $p\not=2$, the counterexample
being given by any smooth
non-linear change of
variable. Here we show that
FIOs of order $m=-d|1/2-1/p|$
are instead bounded.
Moreover, this loss of
derivatives is
 proved to be sharp
in every dimension $d\geq1$, even for phases which are linear in
the dual variables. The proofs make use of tools from
time-frequency analysis such as the theory of modulation spaces.
\end{abstract}

\title{Boundedness of Fourier Integral
Operators on $\Fur L^p$ spaces}\author{Elena Cordero, Fabio Nicola and Luigi Rodino}
\address{Department of Mathematics,
University of Torino, via
Carlo Alberto 10, 10123
Torino, Italy}
\address{Dipartimento di Matematica,
Politecnico di Torino, corso
Duca degli Abruzzi 24, 10129
Torino, Italy}
\address{Department of Mathematics,
University of Torino, via
Carlo Alberto 10, 10123
Torino, Italy}
\email{elena.cordero@unito.it}
\email{fabio.nicola@polito.it}
\email{luigi.rodino@unito.it}
\thanks{}
\maketitle

\section{Introduction}
Consider the spaces $\fpc$ of compactly
supported distributions whose Fourier
transform is in $L^p(\R^d)$, with the
norm $\|f\|_{\Fur
L^p}=\|\hat{f}\|_{L^p}$. Let $\phi$ be
a smooth function. Then it follows from
the Beurling-Helson theorem
(\cite{beurling53}, see also
\cite{lebedev94}) that the map
$Tf(x)=a(x)f(\phi(x))$, $a\in
C^\infty_0(\R^d)$, generally fails to
be bounded on $\fpc$, if $\phi$ is
non-linear. This is of course a Fourier
integral operator of special type,
namely $Tf(x)=\int e^{2\pi
i\phi(x)\eta}
a(x)\hat{f}(\eta)\,d\eta$, by the
Fourier inversion formula.\par In this
paper we study the action on $\fpc$ of
 Fourier Integral
Operators (FIOs) of the form
\begin{equation}\label{FIO}
Tf(x)=\int e^{2\pi
i\Phi(x,\eta)}
\sigma(x,\eta)\hat{f}(\eta)\,d\eta.
\end{equation}
The symbol $\sigma$ is in $S^m_{1,0}$,
the H{\"o}rmander class of order $m$.
Namely, $\sigma\in
\mathcal{C}^\infty(\R^{2d})$ and
satisfies
\begin{equation}\label{symb}
|\partial^\alpha_x\partial^\beta_\eta
\sigma(x,\eta)|\leq
C_{\alpha,\beta}
\langle\o\rangle^{m-|\beta|},
\quad \forall
(x,\o)\in\R^{2d}.
\end{equation}
We suppose that
$\sigma$ has compact support
with respect to $x$.

The phase $\Phi(x,\eta)$ is
real-valued,
 positively
 homogeneous of degree 1 in $\eta$, and smooth on
 $\R^d\times(\R^d\setminus\{0\})$.  It is actually
 sufficient to assume $\Phi(x,\eta)$ defined on
 an open
subset
 $\Lambda\subset \R^d\times(\R^d\setminus\{0\})$,
 conic in dual variables,
containing the points
$(x,\eta)\in{\rm
supp}\,\sigma$, $\eta\not=0$.
More precisely, after setting
\[
\Lambda'=\{(x,\eta)\in\R^d\times(\R^d\setminus\{0\}):
(x,\lambda\eta)\in{\rm
supp}\,\sigma\ \textrm{for
some}\ \lambda>0\},
\]
 we
require that $\Lambda$
contains the closure of
$\Lambda'$ in
$\R^d\times(\R^d\setminus\{0\})$.
We also assume the
non-degeneracy condition
\begin{equation}\label{nondeg}
{\rm det}\,
\left(\frac{\partial^2\Phi}{\partial
x_i\partial \eta_l}\Big|_{
(x,\eta)}\right)\not=0\quad \forall
(x,\eta)\in \Lambda.
\end{equation}
It is easy to see that such
an operator maps the space
$\mathcal{S}(\R^d)$ of
Schwartz functions into the
space
$\mathcal{C}^\infty_0(\R^d)$
of test functions
continuously. We refer to the
books \cite{hormander,treves}
for the general theory of
FIOs, and especially to
\cite{sogge93,stein93} for
results in $L^p$.
\par As the above example
shows, general boundedness
results in $\fpc$ are
expected only for FIOs of
negative order. Our main
result deals precisely with
the minimal loss of
derivatives for boundedness
to hold.
\begin{theorem}\label{maintheorem2}
Assume the above hypotheses
on the symbol $\sigma$ and
the phase $\Phi$. If
\begin{equation}\label{soglia}
m\leq-d\left|\frac{1}{2}-\frac{1}{p}\right|,
\end{equation}
then the corresponding FIO T,
initially defined on
$\mathcal{C}^\infty_0(\R^d)$,
extends to a bounded operator
on $\fpc$, whenever $1\leq
p<\infty$. For $p=\infty$,
$T$ extends to a bounded
operator on the closure of
$\mathcal{C}^\infty_0(\R^d)$
in $\Fur L^\infty(\R^d)_{\rm
comp}$.
\end{theorem}
The loss of derivatives in
\eqref{soglia} is proved to
be sharp in any dimension
$d\geq1$,  even for phases
$\Phi(x,\eta)$ which are
linear in $\eta$ (see Section
\ref{sharp}). In contrast,
notice that  FIOs are
continuous on $L^p(\rd)_{\rm
comp}$ if
\begin{equation*}
m\leq-(d-1)\left|\frac{1}{2}-\frac{1}{p}\right|,
\end{equation*}
and this threshold is sharp. In
particular, for $d=1$, the continuity
is attained without loss of
derivatives, i.e., $m=0$ (\cite[Theorem
2, page 402]{stein93}, see also
\cite{tao}).\par As a model, consider  the following simple example. In dimension $d=1$, take the phase
$\Phi(x,\eta)=\varphi(x)\eta$, where
$\varphi:\R\to\R$ is a diffeomorphism,
 with $\f(x)=x$ for $|x|\geq 1$ and whose
 restriction to $(-1,1)$ is non-linear.
 Consider then the FIO
 \begin{equation}\label{FIOes}
Tf(x)=\int_{\R} e^{2\pi
i\varphi(x)\eta}
G(x)\la \eta\ra^{{m}}\hat{f}(\eta)\,d\eta,
\end{equation}
with $G\in\cC_0^\infty(\R)$, $G(x)=1$ for $|x|\leq 1$.
Let moreover $1\leq p\leq2$. Then
Theorem \ref{maintheorem2} and the
discussion in Section \ref{sharp} below
show that $T:\Fur L^p(\R)_{\rm comp}\to
\Fur L^p(\R)_{\rm comp}$ is bounded if
and only if $m$ satisfies
\eqref{soglia}, with $d=1$.\par The techniques
employed to prove Theorem
\ref{maintheorem2} differ from those
used in \cite[Theorem 2, page
402]{stein93} and \cite{tao} for the
local $L^p$-boundedness. Indeed, in
\cite{stein93} the main idea was to
split the frequency-localized operator,
of order $-(d-1)/2$, into a sum of 
$O(2^{j(d-1)/2})$ FIOs whose phases are
essentially linear in $\eta$ (there
$\langle\eta\rangle\asymp 2^j$), hence
satisfying the desired estimates
without loss of derivatives. Similarly,
in \cite{tao}, the main point was a
decomposition of $T$ into a part with
phase which is non-degenerate with
respect to $\eta$ (and further
factorized) and a degenerate part, with
a phase closer to the linear case,
fulfilling better estimates. However,
as the  example \eqref{FIOes} shows, the case of
linear phases in $\eta$ already
contains all the obstructions to the
local $\Fur L^p$-boundedness, so that we cannot here take
advantage of this kind of
decompositions. Instead, the proof of
Theorem \ref{maintheorem2} makes use of
the theory of modulation spaces $M^p$,
$1\leq p\leq \infty$, which are now
classical function spaces used in
time-frequency analysis (see
\cite{fe89-1,feichtinger90,grochenig}
and Section 2  for definition and
properties).\par In short, we say that
a temperate distribution $f$ belongs to
$M^p(\R^d)$ if its short-time Fourier
transform $V_g f\phas$, defined in
\eqref{STFT} below, is in $L^p(\rdd)$.
Then, $\|f\|_{M^p}:=\|V_g f\|_{L^p}$.
Here $g$ is a non-zero (so-called
window) function in $\cS(\rd)$, and
changing $g\in \cS(\rd)$ produces
equivalent norms. The space
$\mathcal{M}^{\infty}(\rd)$ is
 the closure of $\cS(\rd)$ in the
 $M^\infty$-norm. For heuristic purposes,
distributions in $M^p$ may be
regarded as functions which
are locally in ${\Fur L^p}$
and decay at infinity like
a function in $L^p$.
\par Modulation spaces are
relevant here since, for distributions
supported in a fixed compact subset
$K\subset\R^d$, there exists $C_K>0$
such that
\[
C_K^{-1}\|u\|_{M^p}\leq\|u\|_{\Fur
L^p}\leq C_K\|u\|_{M^p},
\]
(see
\cite{fe89-1,feichtinger90,grobner} for
even more general embeddings, and the
subsequent Lemma \ref{lloc}). Then, our
framework can be shifted from $\cF L^p$
to $M^p$ spaces and, using techniques
from time-frequency analysis,  Theorem
\ref{maintheorem2} shall be proven in a
slightly wider  generality:
\begin{theorem}\label{maintheorem}
Assume the above hypotheses
on the symbol $\sigma$ and
the phase $\Phi$. Moreover,
assume \eqref{soglia}. Then,
if $1\leq p<\infty$, the
corresponding FIO $T$,
initially defined on
$\mathcal{S}(\R^d)$, extends
to a bounded operator on
$M^p$; if $p=\infty$, the
operator $T$  extends to a
bounded operator on
$\mathcal{M}^\infty$.
\end{theorem}
 Since the
proof of Theorem \ref{maintheorem} is quite technical, for the
benefit of the reader we first
 exhibit the general pattern. We start by splitting up the symbol
$\sigma(x,\eta)$ of \eqref{FIO} into the  sum of two symbols
supported where $|\eta|\leq
4$ and $|\eta|\geq2$, respectively. These
symbols give rise to two
FIOs $T_{1}$ and $T_{2}$ which will be
studied separately.\par
 The operator $T_{1}$ carries the singularity of the phase $\Phi$ at the origin and is  proved to be
bounded on $M^1$ and on
$\mathcal{M}^\infty$. The
boundedness on $M^p$, for
$1<p<\infty$, follows from
complex interpolation.
Precisely, the boundedness on
 $M^1$ is straightforward.
For the $\mathcal{M}^\infty$-case, we use the fact
 that $e^{2\pi
 i\Phi(x,\eta)}\chi(\eta)$,
$\chi$ being a cut-off
function which localizes
 near the origin,
 is in $\Fur L^1_\eta$,
 uniformly with respect to
 $x$ (cf. Theorem \ref{bnj}).\par
For the operator $T_{2}$, which carries the oscillations at infinity,
 we still prove the
desired boundedness result on $M^1$ and on $\mathcal{M}^\infty$ in
the case $m=-d/2$. Here the general case follows by  interpolation
with the well-known case $M^2=L^2$, $m=0$, see e.g. \cite[page
397]{stein93}.\par To chase our goal, we perform a dyadic
decomposition of the symbol (of order $m=-d/2$) in shells where
$|\eta|\asymp 2^j$, $j\geq0$, obtaining the representation
$T_{2}=\sum_{j=1}^\infty T^{(j)}$. Each dyadically localized
operator $T^{(j)}$ is then  conjugated with the dilations
operators $U_{2^{j/2}}f(x)=f(2^{j/2}x)$, so that
\begin{equation}\label{i1000}
T^{(j)}=U_{2^{j/2}}\tilde{T}^{(j)}
U_{2^{-j/2}},
\end{equation}
where $\tilde{T}^{(j)}$ is
a FIO with phase
\begin{equation}\label{i100}
{\Phi}_j(x,\eta)=\Phi(2^{-j/2}x,2^{j/2}\eta)=2^{j/2}\Phi(2^{-j/2}x,\eta).
\end{equation}
Now, ${\Phi}_j(x,\eta) $ has
derivatives of order $\geq2$
bounded on the support of the
corresponding symbol, and, as a consequence,
the operators
$\tilde{T}^{(j)}$ are bounded
on $M^1$ and
$\mathcal{M}^\infty$, with
operator norm $\lesssim
2^{-jd/2}$.
These results are established and proved   in
Propositions \ref{pro1} and
\ref{pro2}; they can be seen as a microlocal
version of  \cite{toft07,fio1},
 where $M^p$-boundedness (without loss of derivatives)
  was proved for FIOs with phase function possessing
   globally bounded derivatives
   of order $\geq 2$ (extending results for $p=2$ in
   \cite{boulkhemair}).
 Combining this
with  boundedness results
 for dilation operators on
modulation spaces
\cite{sugimototomita}, we
obtain the boundedness of
$T^{(j)}$ uniformly with
respect to $j$. In this way,
the assumption \eqref{soglia}
is used to compensate the
bound for the norm of the
dilation operator.\par A last
non-trivial technical problem
is  summing (on $j\geq1$) the
corresponding estimates. For
this, we make use of the
almost orthogonality of the
$T^{(j)}$. As a tool, for
$M^1$ we will need an equivalent
characterization of the
$M^1$-norm, see \eqref{normmp} below.
On the other hand,  for the
$\mathcal{M}^\infty$-case, we
 essentially use the following property:

\textit{If $\hat{u}$ is
localized in the shell
$|\eta|\asymp 2^j$, then
$\widehat{Tu}$ is localized
in the neighbour shells}.

\noindent In practice, this
is achieved by means of
another dyadic partition in
the frequency domain and the
composition formula for a
pseudodifferential operator
and a FIO (see Theorem
\ref{composition}).\par
Finally, we observe that the
above trick of conjugating
with dilations has a nice
interpretation in terms of
the geometry of the symbol
classes, namely, the
associated partition of the
phase space by suitable
boxes. The  symbol estimates
for the H{\"o}rmander's
classes correspond to a
partition of the phase space
in boxes of size $1\times
2^j$, $j\geq0$. Instead,  the
estimates satisfied by the
phases ${\Phi}_j(x,\eta)$ in
\eqref{i100}, are of Shubin's
type \cite{Shubin91}, and
correspond to a partition of
the phase space by boxes of
size $2^{j/2}\times 2^{j/2}$.
This enters the general philosophy of \cite{fefferman}.  Figure 1
shows the passage from the
H{\"o}rmander's geometry to
the Shubin's one, performed
in \eqref{i1000}.

\vspace{1.2cm}
 \begin{center}
           \includegraphics{figFIO.1}
            \\
           $ $
\end{center}

 \begin{center}{ Figure 1. }
           \end{center}
 As a motivation of our study we recall that
 the solutions of the Cauchy problem for a strictly
  hyperbolic equation can be described by Fourier
   integral operators. If the equation has constant
   coefficients, then $T$ reduces to a Fourier
    multiplier and acts continuously on $M^p$
    without loss of derivatives,
     cf. \cite{benyi,toft07,fio1}. In the case of
     variable coefficients, our Theorem \ref{maintheorem}
      gives the optimal regularity results on $M^p$.
      One could be disappointed then,  since the loss
       $d|1/2-1/p|$ is even larger than  $(d-1)|1/2-1/p|$
        in $L^p$, cf. \cite{sogge93,stein93}.
        However, when passing to treat
FIOs  whose symbols are not
compactly supported in $x$, cf. \cite{ruzhsugimoto}, modulation spaces seem
 really
to be the right function
spaces to control both local
regularity and decay at
infinity. We plan to devote a
subsequent paper  to such
topics.\par\medskip The paper
is
organized as follows.
In Section \ref{section2}  the definitions and
basic properties of the spaces
$\Fur L^p$ and the modulation
spaces $M^p$ are reviewed.  Section
\ref{section3} contains preliminaries on  FIOs and the proof of the above mentioned
microlocal version of results
in \cite{fio1}. In Section \ref{section4}
we prove Theorem
\ref{maintheorem} for a
symbol $\sigma(x,\eta)$ which,
in addition, vanishes for
$\eta$ large.  Section
\ref{section5} is devoted to the proof of
Theorem \ref{maintheorem} for
a symbol $\sigma(x,\eta)$
 vanishing
for $\eta$ small. Finally,
Section \ref{sharp} exhibits
the optimality of both
Theorems \ref{maintheorem2}
and \ref{maintheorem}.

\vskip0.5truecm

\par
\textbf{Notation.} We define
$|x|^2=x\cdot x$, for
$x\in\Ren$, where $x\cdot
y=xy$ is the scalar product
on $\Ren$. The space of
smooth functions with compact
support is denoted by
$\cC_0^\infty(\rd)$, the
Schwartz class is
$\sch(\Ren)$, the space of
tempered distributions
$\sch'(\Ren)$.    The Fourier
transform is normalized to be
${\hat
  {f}}(\o)=\Fur f(\o)=\int
f(t)e^{-2\pi i t\o}dt$.
 Translation and modulation operators ({\it time and frequency shifts}) are defined, respectively, by
$$ T_xf(t)=f(t-x)\quad{\rm and}\quad M_{\o}f(t)= e^{2\pi i \o
 t}f(t).$$
We have the formulas
$(T_xf)\hat{} = M_{-x}{\hat
{f}}$, $(M_{\o}f)\hat{}
=T_{\o}{\hat {f}}$, and
$M_{\o}T_x=e^{2\pi i
x\o}T_xM_{\o}$. The inner
product of two functions
$f,g\in \lrd$ is $\la f, g
\ra=\intrd f(t)
\overline{g(t)}\,dt$, and its
extension to $\cS'\times\cS$
will be also denoted by  $\la
\cdot, \cdot \ra$.The
notation $A\lesssim B$ means
$A\leq c B$ for a suitable
constant $c>0$, whereas $A
\asymp B$ means $c\inv A \leq
B \leq c A$, for some $c\geq
1$. The symbol $B_1
\hookrightarrow B_2$ denotes
the continuous embedding of
the space $B_1$ into $B_2$.
\section {Function spaces and
preliminaries}\label{section2}
\subsection{$\Fur L^p$ spaces}
For every $1\leq
p\leq\infty$, we define by
 $\cF L^p$ the Banach space of all
 distributions $\hat{f}\in\cS'$ such that
$f\in L^p(\rd)$
endowed with the norm
$$
\|\hat{f}\|_{\cF
L^p}=\|f\|_{L^p};
$$
for details see, e.g.,
\cite{kat}.\par The space
$\Fur L^p(\R^d)_{\rm comp}$
consists of the
distributions in $\Fur
L^p(\R^d)$ having compact
support. It is the inductive
limit of the Banach spaces
$\Fur L^p(K_n):=\{f\in \Fur
L^p(\R^d):\ {\rm supp}\,
f\subset K_n\}$, where
$\{K_n\}$ is any increasing
sequence of compacts whose
union is $\R^d$.\par Since
the operator $T$ in the
Introduction has a symbol
compactly supported in $x$,
we see that the conclusion of
Theorem \ref{maintheorem2} is
equivalent to an estimate of
the type
\[
\|Tu\|_{\Fur L^p}\leq
C_K\|u\|_{\Fur L^p},\quad
\forall
u\in\mathcal{C}^\infty_0(K),
\]
for every compact
$K\subset\R^d$.
\subsection{Modulation
spaces}(\cite{fe89-1,feichtinger90,grochenig}).
Let $g\in\cS$ be a non-zero
window function. The
short-time Fourier transform
(STFT) $V_gf$ of a
function/tempered
distribution $f$ with respect
to the the window $g$ is
defined by
\begin{equation}\label{STFT}
V_g f(x,\o)=\la f,M_\o T_xg\ra= \int e^{-2\pi i \o y}f(y)\overline{g(y-x)}\,dy,
\end{equation}
i.e.,  the  Fourier transform
$\cF$ applied to
$f\overline{T_zg}$.\par There
is also an inversion formula
for the STFT (see e.g.
(\cite[Corollary
3.2.3]{grochenig}). Namely,
if $\|g\|_{L^2}=1$ and, for
example, $u\in{L^2}(\R^d)$,
it turns out
\begin{equation}\label{treduetre}
u=\int_{\R^{2d}} V_g
u(y,\o)M_\o T_y g\,
dy\,d\o.
\end{equation}

\par For
$1\leq p\leq\infty$,
$s\in\R$, the modulation
space $M^{p}_s(\R^n)$ is
defined as the space of all
distributions $f\in\cS'(\rd)$
such that the norm
\[
\|f\|_{M^{p}_s(\rd)}
=\left(\intrdd|V_g
f\phas|^p\langle\eta\rangle^{sp}
\,dxd\o\right)^{1/p}
\] is
finite (with the obvious
changes if $p=\infty$). If
$s=0$ we simply write $M^p$
in place of $M^p_0$. This
definition is independent of
the choice
 of the  window $g$ in the sense of equivalent norms.
  Moreover, if $1\leq p<\infty$, $M^1_s$ is densely
   embedded into $M^p_s$, as is the Schwartz class $\cS$.
     Among the properties of \modsp s, we record that
$M^{2}=L^2$, $M^1\subset L^1\cap\cF L^1$.
 For $p\not=2$, $M^p$ does not coincide
 with any Lebesgue space.
  Indeed, for $1\leq p<2$, $M^p\subset L^p$, whereas,
   for $2<p\leq\infty$, $L^p\subset M^p$. Differently
    from the $L^p$ spaces, they enjoy the embedding
     property:
 $M^{p}\hookrightarrow M^{q}$, if $p\leq
q$. If $p<\infty$, the dual
of $M^{p}_s$ is
$(M^{p}_s)'=M^{p'}_{-s}$.\par
Let us define by
$\mathcal{M}^p_s(\rd)$ the
completion of $\cS(\rd)$
under the norm
 $\|\cdot\|_{M^p_s}$. Then the following are true \cite{Feichtinger-grochenig89}:\\
(i) If $1\leq p<\infty$, then
$\mathcal{M}^p_s(\rd)= M^p_s(\rd)$,\\
(ii) If $1\leq p_1,p_2\leq
\infty$, $s_1,s_2\in\R$,
 and $0<\theta<1$, $1\leq p\leq\infty$, $s\in\R$
 satisfy
 $1/p=(1-\theta)/p_1+\theta/p_2$,
  $s=(1-\theta)s_1+\theta s_2$, then
\begin{equation}\label{interpmp}
(\mathcal{M}^{p_1}_{s_1},
\mathcal{M}^{p_2}_{s_2})_{[\theta]}=\mathcal{M}^p_s,
\end{equation}
(iii)
$(\mathcal{M}^\infty_s(\rd))'=M^1_{-s}$.

In the sequel it will be
useful the following
characterization of the
$M^p_s$ spaces (see, e.g.,
\cite{feichtinger90,triebel83}):
 let
$\f\in\cC^{\infty}_0(\rd)$, $\f\geq
0$,  such that $\sum_{m\in\zd}
\f(\eta -m)\equiv 1$, for all
$\o\in\rd$. Then
\begin{equation}\label{normmp}
\|u\|_{M^p_s}\asymp
\left(\sum_{m\in\zd} \|\f(D
-m)u\|_{L^p}^p\langle
m\rangle^{ps}\right)^{1/p},
\end{equation}
where $\f(D -m)u =\cF^{-1}[\f(\cdot
-m)\hat{u}]$ (with the obvious
changes if $p=\infty$).

If we consider the
space of functions/distributions $u$
 in $M^p(\rd)$ that are  supported in any fixed
  compact set,
  then their $M^p$-norm is equivalent
   to the $\cF L^p$-norm.  More precisely, we have the following result \cite{fe89-1,feichtinger90,grobner,kasso07}:
\begin{lemma}\label{lloc} Let $1\leq p\leq \infty.$
For every $u\in
\mathcal{S}'(\rd)$, supported
in a compact set $K\subset
\rd$, we have $u\in
M^p\Leftrightarrow u\in \Fur
L^p$, and
\begin{equation}\label{loc}
C_K^{-1} \|u\|_{M^p}\leq \|u\|_{\cF L^p}\leq C_K \|u\|_{M^p},
\end{equation}
where $C_K>0$ depends only on
 $K$.
\end{lemma}

In order to state the dilation properties for modulation spaces, we introduce the indices:
$$ \mu_1(p)=\begin{cases}-1/{p^\prime} &  \quad {\mbox{if}} \quad 1\leq p\leq 2,\\
 -1/p &   \quad {\mbox{if}}   \quad  2\leq p\leq
 \infty,
 \end{cases}
 $$
and
$$ \mu_2(p)=\begin{cases}-1/p &  \quad {\mbox{if}} \quad 1\leq p\leq 2,\\
 -1/{p^\prime} &   \quad {\mbox{if}}  \quad   2\leq p\leq \infty.\\
 \end{cases}
 $$
For $\lambda>0$, we define
the dilation operator
$U_\lambda f(x)=f(\lambda
x)$.  Then, the dilation
properties of $M^p$ are as
follows (see \cite[Theorem
3.1]{sugimototomita}).

\begin{theorem}\label{dilprop}
We have:
 (i) For $\lambda\geq 1 $,
$$\| U_\lambda f\|_{M^p}\lesssim \lambda^{d\mu_1(p)}
\|f\|_{M^p},\quad\forall\,
f\in M^p(\rd).
$$
(ii) For $0<\lambda\leq 1 $,
$$\|  U_\lambda f\|_{M^p}\lesssim \lambda^{d\mu_2(p)}
\|f\|_{M^p},\quad\forall \,
f\in M^p(\rd).
$$
\end{theorem}

\noindent These dilation
estimates are sharp, as
discussed in
\cite{sugimototomita}, see also \cite{cordero2}.
\section{Preliminary results on
FIOs}\label{section3} In this
section we recall the
composition formula of a
pseudodifferential operator
and a FIO. Then we prove some
auxiliary results for FIOs
with phases having bounded
derivatives of order $\geq2$.
\subsection{Composition of
pseudodifferential and Fourier integral operators}
First, recall the general H\"ormander symbol class
$S^m_{\rho,\delta}$ of smooth
functions on $\rdd$ such that
$$|\partial^{\alpha}_x\partial^{\beta}_\o
\sigma\phas|\leq C_{\a,\beta}
\la\o\ra^{m-\rho
|\beta|+\delta
|\alpha|},\quad \phas\in\rdd.
$$
 A regularizing operator
is a pseudodifferential
operator
\[
Ru=\int e^{2\pi
ix\eta}r(x,\eta)\hat{u}(\eta)d\eta,
\]
with a symbol $r$ in the
Schwartz space
$\mathcal{S}(\R^{2d})$
(equivalently, an operator with kernel in $\mathcal{S}(\R^{2d})$,
which maps
$\mathcal{S}'(\R^{d})$ into
$ \mathcal{S}(\R^{d})$).
Then, the composition formula
for a pseudodifferential
operator and a FIO is as
follows (see, e.g., \cite{hormander},
\cite[Theorem
4.1.1]{mascarello-rodino},
\cite[Theorem
18.2]{Shubin91},
\cite{treves}; we limit ourselves to recall what is needed in the subsequent proofs).
\begin{theorem}\label{composition}
Let the symbol $\sigma$ and
the phase $\Phi$ satisfy the
assumptions in the
Introduction. Assume, in
addition, $\sigma(x,\eta)=0$
for $|\eta|\leq1$, if
$\Phi(x,\eta)$ is not linear
in $\eta$. Let $a(x,\eta)$ be
a symbol in $S^{m'}_{1,0}$.
Then,
\[
a(x,D)T=S+R,
\]
where $S$ is a FIO with the same phase
$\Phi$ and symbols $s(x,\eta)$, of
order $m+m'$, satisfying
\[
{\rm supp}\,s\subset {\rm
supp}\,\sigma\cap\{(x,\eta)\in\Lambda:\
(x,\nabla_x\Phi(x,\eta))\in{\rm
supp}\,a\},
\]
and $R$ is a regularizing
operator with symbol
$r(x,\eta)$ satisfying
\[
\Pi_\eta ({\rm
supp}\,r)\subset \Pi_\eta(
{\rm supp}\,\sigma),
\]
where $\Pi_\eta$ is the orthogonal
projection on $\R^d_\eta$.\par Moreover,
the symbol estimates satisfied by $s$
and the seminorm estimates of $r$ in
the Schwartz space are uniform when
$\sigma$ and $a$ vary in a bounded
subsets of $S^m_{1,0}$ and $S^{m'}_{1,0}$
respectively.
\end{theorem}

\subsection {FIOs with phases
having bounded derivatives of
order $\geq2$} In what
follows we present a
micro-localized version of
\cite[Theorems 3.1,
4.1]{fio1}, where the
hypotheses of such theorems
are satisfied only in the
$\epsilon$-neighborhood
$\Sigma_\epsilon$ of the
support of $\sigma$, $\sigma$
being the FIO's symbol.
Namely, set
$$\Sigma_\epsilon=\cup_{(x_0,\eta_0)\in{\rm
supp}\,\sigma} B_\epsilon
(x_0,\eta_0).$$

\begin{proposition}\label{pro1}
Let $\sigma\in S^0_{0,0}$ and
$\Sigma_\epsilon$ as above.
Let $\Phi$ be a real-valued function
defined and smooth on
$\Sigma_\epsilon$. Suppose
that
\begin{equation}\label{deriv}
|\partial^\alpha_z
\Phi(z)|\leq C_\alpha\quad
{\it for}\ |\alpha|\geq2,\
z=(x,\eta)\in\Sigma_\epsilon.
\end{equation}
Let
$g,\gamma\in\mathcal{S}(\R^d)$,
$\|g\|_{L^2}=
\|\gamma\|_{L^2}=1$, with
${\rm supp}\, \gamma\subset
B_{\epsilon/4}(0)$, ${\rm
supp}\, \hat{g}\subset
B_{\epsilon/4}(0)$. Then, for
every $N\geq0$, there exists
a constant $C>0$ such that
\[
|\langle T (M_\omega T_y g),
M_{\omega'}
T_{y'}\gamma\rangle|\leq
C\mathds{1}_{\Sigma_{\epsilon/2}}(y',\omega)\langle
\nabla_x\Phi(y',\omega)-\omega'\rangle^{-N}\langle
\nabla_\eta\Phi(y',\omega)-y\rangle^{-N}.
\]
The constant $C$ only depends
on $N$, $g,\gamma$, and upper
bounds for a finite number of
derivatives of $\sigma$ and
on a finite number of
constants in \eqref{deriv}.
\end{proposition}
\begin{proof}
 We can write
\begin{align*}
  \langle T &(M_\omega T_y g),
M_{\omega'} T_{y'}\gamma\rangle \\
   =& \intrd T M_{\omega}T_{y}g(x)
   \overline{M_{\omega'}T_{y'}\gamma(x)}\,dx \\
   =& \intrd\intrd e^{2\pi i \Phi\phas}\sigma\phas
    T_{\omega}M_{-y}\hat{g}(\o)M_{-\omega'}T_{y'}
    \bar{\gamma}(x)\,dx
   d\o\\
   =&\intrd\intrd M_{(0,-y)}T_{(0,-\omega)}\left(e^{2\pi i \Phi\phas}\sigma\phas\right) \hat{g}(\o)M_{-\omega'}T_{y'}\bar{\gamma}(x)\,dx
   d\o\\
   =&\intrd\intrd T_{(-y',0)} M_{-(\omega',0)}
   M_{(0,-y)}T_{(0,-\omega)}\left(e^{2\pi i \Phi\phas}
   \sigma\phas\right) \bar{\gamma}(x)\hat{g}(\o)\,dx
   d\o\\
    =&\intrd\intrd e^{2\pi i [\Phi(x+y',\eta+\omega)-
    (\omega',y)\cdot (x+y',\o)]}\sigma(x+y',\eta+\omega)
    \bar{\gamma}(x)\hat{g}(\o)\,dx
   d\o\\
   =&\int_{B_{\epsilon/4}(0)}\int_{B_{\epsilon/4}(0)}
e^{2\pi i
[\Phi((x+y',\eta+\omega))-(\omega',y)\cdot
(x+y',\o)]}\sigma(x+y',\eta+\omega)
\bar{\gamma}(x)\hat{g}(\o)\,dx
   d\o.
\end{align*}
Observe that, if
$(y',\omega)\not\in\Sigma_{\epsilon/2}$,
 by the assumptions on the
support of $\gamma$ and
$\hat{g}$ and the triangle
inequality, this integral
vanishes.\\
 Hence, assume
$(y',\omega)\in
\Sigma_{\epsilon/2}$. Since
$\Phi$ is smooth on
$\Sigma_\epsilon$, we perform
a Taylor expansion of
$\Phi\phas$ at $(y',\omega)$
and obtain
$$\Phi(x+y',\eta+\omega)=\Phi(y',\omega)+
\nabla_z\Phi(y',\omega)\cdot
(x,\o)+
\Phi_{2,(y',\omega)}\phas,
$$
for $z=(x,\eta)\in
B_{\epsilon/4}(0)\times
B_{\epsilon/4}(0)$, where the
remainder is given by
 \begin{equation}
  \label{eq:c11}
\Phi_{2,(y',\omega)}\phas=2\sum_{|\a|=2}\int_0^1(1-t)
\partial^\a
\Phi((y',\omega)+t\phas)\,dt\frac{\phas^\a}{\a!}.
\end{equation}
Notice that the segment
$(y',\omega)+t\phas$, $0\leq
t\leq1$, belongs entirely to
$\Sigma_\epsilon$ if
$(x,\eta)\in
B_{\epsilon/4}(0)\times
B_{\epsilon/4}(0)$.\\ Whence,
we can write
\begin{align*}
  |\langle& T (M_\omega T_y g),
M_{\omega'} T_{y'}\gamma\rangle|
    =\Big|\int_{B_{\epsilon/4}(0)}\int_{B_{\epsilon/4}(0)}
    e^{2\pi i \{[\nabla_z\Phi(y',\omega)-(\omega',y)]
    \cdot
     \phas\}}\\
    &\qquad\qquad\qquad\times e^{2\pi i\Phi_{2,
    (y',\omega)}\phas} \sigma(x+y',\eta+\omega)
    \bar{\gamma}(x)\hat{g}(\o)\,dx
    d\o\Big|.
\end{align*}
For $N\in\bN$, using the identity:
\begin{multline*}
(1-\Delta_x)^N(1-\Delta_\eta)^N
 e^{2\pi i \{[\nabla_z\Phi(y',\omega)-(\omega',y)]\cdot
  \phas\}}=\la
    2\pi(\nabla_x\Phi(y',\omega)-\omega')\rangle^{2N}
   \\
    \times
\langle2\pi(\nabla_\eta\Phi(y',\omega)-y)\rangle^{2N}e^{2\pi
i
\{[\nabla_z\Phi(y',\omega)-(\omega',y)]\cdot
  \phas\}},
\end{multline*} we integrate by parts and
obtain
\begin{align*}
  |\langle T (M_\omega T_y g),
M_{\omega'}
T_{y'}\gamma\rangle|&
    =\langle2\pi(\nabla_x\Phi(y',\omega)-\omega')\rangle^{-2N}
\langle2\pi(\nabla_\eta\Phi(y',\omega)-y)\rangle^{-2N}\\
\times\Big|\int_{B_{\epsilon/4}(0)}\int_{B_{\epsilon/4}(0)}
&e^{2\pi i
\{[\nabla_z\Phi(y',\omega)-(\omega',y)]\cdot
    \phas\}}\\
    \times (1-\Delta_x)^N & (1-\Delta_\eta)^N
    [e^{2\pi
    i\Phi_{2,(y',\omega)}\phas}
\sigma(x+y',\eta+\omega)
    \bar{\gamma}(x)\hat{g}(\o)]\,dx   d\o\Big|.
\end{align*}
Hence it suffices to apply
the Leibniz formula taking
into account that, as a
consequence of \eqref{deriv},
we have the estimates
$\partial_z^{\alpha}
\Phi_{2,(y',\omega)}(z)=O(\langle
z\rangle^{2})$ for
$z=(x,\eta)\in
B_{\epsilon/4}(0)\times
B_{\epsilon/4}(0)$, uniformly
with respect to
$(y',\omega)$.
\end{proof}

In the following proposition, where supp $\sigma$ and $\Sigma_\epsilon$ are understood to be bounded in the $\eta$ variables, we prove the $M^p$-continuity of $T$ with uniform norm bound with respect to the constants $C_\alpha$ in \eqref{deriv}.

\begin{proposition}\label{pro2}
Consider a symbol $\sigma\in
S^0_{0,0}$, with
$\sigma(x,\eta)=0$ for
$|\eta|\leq 2$. Let moreover
$\Omega\subset\R^d$ open and
$\Gamma\subset\R^d\setminus\{0\}$
conic and open, such that
${\Omega}\times\Gamma$
contains the
$\epsilon$-neighborhood
$\Sigma_\epsilon$ of ${\rm
supp}\,\sigma$.\par Consider
then a phase $\Phi\in
C^\infty({\Omega}\times\Gamma)$,
positively homogeneous of
degree 1 in $\eta$,
satisfying
\begin{equation}\label{i0tris}
|\partial^\alpha_z\Phi(x,\eta)|\leq
C_\alpha\quad {\rm for}\
|\alpha|\geq2,\
(x,\eta)\in\Sigma_\epsilon,
\end{equation}
\begin{equation}\label{i1tris}
\left|{\rm det}\,
\left(\frac{\partial^2\Phi}{\partial
x_i\partial \eta_l}\Big|_{
(x,\eta)}\right)\right|\geq\delta>0,\quad
\forall
(x,\eta)\in{\Omega}\times\Gamma,
\end{equation}
and such that
\begin{equation}\label{i2tris}
\forall x\in{\Omega},\
\textrm{the map}\
\Gamma\ni\eta\mapsto
\nabla_x\Phi(x,\eta)\
\textrm{is a diffeomorphism
onto the range},
\end{equation}
\begin{equation}\label{i3tris}
\forall \eta\in\Gamma,\
\textrm{the map}\ {\Omega}\ni
x\mapsto
\nabla_\eta\Phi(x,\eta)\
\textrm{is a diffeomorphism
onto the range}.
\end{equation}
Then, for every $1\leq
p\leq\infty$ it turns out
\[
\|Tu\|_{M^p}\leq
C\|u\|_{M^p},\quad \forall
u\in\mathcal{S}(\R^d),
 \]
where the constant $C$
depends only on
$\epsilon,\delta$, on upper
bounds for a finite number of
derivatives of $\sigma$ and a
finite number of the
constants in \eqref{i0tris}.
\end{proposition}
\begin{proof}
Let
$g,\gamma\in\mathcal{S}(\R^d)$,
with
$\|g\|_{L^2}=\|\gamma\|_{L^2}=1$,
${\rm supp}\, \gamma\subset
B_{\epsilon/4}(0)$, ${\rm
supp}\, \hat{g}\subset
B_{\epsilon/4}(0)$. Let
$u\in\mathcal{S}(\R^d)$. The
inversion formula
\eqref{treduetre} for the
STFT gives
\[
V_\gamma
(Tu)(y',\omega')=\int_{\R^{2d}}\langle
T(M_\omega T_y g),
M_{\omega'}
T_{y'}\gamma\rangle V_g
u(y,\omega)dy\,d\omega.
\]
The desired estimate follows
if we prove that the map
$K_T$ defined by
\[
K_T
G(y',\omega')=\int_{\R^{2d}}\langle
T(M_\omega T_y g),
M_{\omega'}
T_{y'}\gamma\rangle
G(y,\omega)dy\,d\omega
\]
is continuous on
$L^p(\R^{2d})$. By Schur's
test (see e.g. \cite[Lemma
6.2.1]{grochenig}) it
suffices to prove that its
integral kernel
\[
K_T(y',\omega';y,\omega)=\langle
T(M_\omega T_y g),
M_{\omega'}T_{y'}\gamma\rangle
\]
satisfies
\begin{equation}\label{schur1}
K_T\in
L^\infty_{y',\omega'}(L^1_{y,\omega}),
\end{equation}
and
\begin{equation}\label{schur2}
K_T\in
L^\infty_{y,\omega}(L^1_{y',\omega'}).
\end{equation}
Let us verify \eqref{schur1}.
By Proposition \eqref{pro1}
and the fact that
$\mathds{1}_{\Sigma_{\epsilon/2}}(y',\omega)\leq
\mathds{1}_{{\Omega}}(y')
\mathds{1}_{\Gamma}(\omega)$
we have
\[
|K_T(y',\omega';y,\omega)|\leq
C\mathds{1}_{{\Omega}}(y')
\mathds{1}_{\Gamma}(\omega)\langle
\nabla_x\Phi(y',\omega)-\omega'\rangle^{-N}\langle
\nabla_\eta\Phi(y',\omega)-y\rangle^{-N},\quad
\forall N\in\mathbb{N}.
\]
Hence \eqref{schur1} will be
proved if we verify that
there exists a constant $C>0$
such that
\[
\int\mathds{1}_{\Gamma}(\omega)\langle
\nabla_x\Phi(y',\omega)-\omega'\rangle^{-N}\,d\omega\leq
C,\quad \forall
(y',\omega')\in{\Omega}\times\R^d.
\]
In order to prove this
estimate we perform the
change of variable
\[
\beta_{y'}:\Gamma\ni\omega\longmapsto\nabla_x\Phi(y',\omega),
\]
which is a diffeomorphism on
the range by \eqref{i2tris}.
The Jacobian determinant of
its inverse is homogeneous of
degree 0 in $\omega$ and
uniformly bounded with
respect to $y'$ by the
hypotheses \eqref{i0tris} and
\eqref{i1tris}. Hence, the
last integral is, for $N>d$,
\begin{align*}
&\lesssim\int_{\beta_{y'}(\Gamma)}
\langle\tilde{\omega}-\omega'\rangle^{-N}\,d\tilde{\omega}\\
&\leq \int_{\R^{d}}
\langle\tilde{\omega}-\omega'\rangle^{-N}\,d\tilde{\omega}=C.
\end{align*}
The proof of \eqref{schur2}
is analogous and
left to the reader.\par
Finally, the uniformity of the
norm of $T$ as a bounded
operator, established in the
last part of the statement,
follows from the proof
itself.
\end{proof}

\section{Singularity at the
origin} \label{section4} In
this section we prove Theorem
\ref{maintheorem} for an
operator satisfying the
assumptions stated there and
whose symbol $\sigma$
satisfies, in addition,
\begin{equation}
\sigma(x,\eta)=0\quad {\rm
for}\ |\eta|\geq4.
\end{equation}
Here we do not use the
hypothesis \eqref{nondeg}.
Indeed, we will deduce the
desired result from the
following one, after
extending $\Phi|_{\Lambda'}$
to a phase function, still
denoted by $\Phi(x,\eta)$,
positively homogeneous of
degree 1 in $\eta$, (possibly
degenerate) and  everywhere
defined in
$\R^d\times(\R^d\setminus\{0\})$.
\begin{proposition}\label{proposizione} Let
$\sigma(x,\eta)$ be a smooth
symbol satisfying
\begin{equation}\label{sigmaloc}
\sigma(x,\eta)=0\quad {\rm
for}\ |x|+|\eta|\geq R,
\end{equation}
for some $R>0$. Let
$\Phi(x,\eta)$, $x\in\R^d$,
$\eta\in\R^d\setminus\{0\}$,
be a smooth phase function,
positively homogeneous of
degree 1 in $\eta$. Then the
corresponding FIO $T$ extends
to a bounded operator on
$M^p$, for every $1\leq
p<\infty$, and on
$\mathcal{M}^\infty$.
\end{proposition}
In order to prove Proposition
\ref{proposizione}, we use
the complex interpolation
method between the spaces
$M^1$ and
$\mathcal{M}^\infty$. Indeed,
using \eqref{interpmp},
 $$(M^1, \mathcal{M}^\infty)_{[\theta]}=M^p,\quad \frac1p=\theta,\quad 0<\theta<1,$$
 we attain the boundedness of $T$ on every $M^p$,
 $1<p<\infty$, if we prove the boundedness on $M^1$ and
 $\mathcal{M}^\infty$. The
 rest of this section is
 devoted to that.

\subsection{Boundedness on
$M^1$.} For every
$u\in
\cS(\rd)$, we have
$$ \partial^\a (Tu)(x)=\intrd e^{2\pi i \Phi(x,\o)}
\left(\sum _{\beta\leq \a}
 p_\beta(\partial^{|\beta|}\Phi(x,\eta))
 \partial_x^{\a-\beta}\sigma(x,\eta)\right)
 \hat{u}(\o)\,d\o,
$$
where
$p_\beta(\partial^{|\beta|}\Phi(x,\eta))$
is a polynomial of order
$|\beta|$ in the derivatives
of $\Phi$ with respect to $x$
of order at most $|\beta|$.\\
Hence, because of the
homogeneity of $\Phi$ and
\eqref{sigmaloc},
$$\|
\partial^\a (Tu)\|_{L^1}\leq
\left(\sum _{\beta\leq \a}
\int_{|x|\leq R}
\tilde{C}_{\beta}
\sup_{|\eta|\leq R}
\langle\eta\rangle^{|\beta|}
|\partial_x^{\a-\beta}\sigma(x,\eta)|
dx \right)
\|\hat{u}\|_{L^1}\leq C_\a
\|{u}\|_{\cF L^1}.
$$

Using the relation \eqref{loc}, the
previous estimate, and  the inclusion
$M^1\hookrightarrow\cF L^1$, the result
is easily attained:

$$\| Tu\|_{M^1}\asymp\| Tu\|_{\cF L^1}\leq C \sup_{|a|\leq d+1}\|\partial^\a (Tu)\|_{ L^1}\leq C \sup_{|a|\leq d+1} C_\a \|{u}\|_{\cF L^1}\leq \tilde{C} \|{u}\|_{M^1}.
$$
\subsection{Boundedness on
$\mathcal{M}^\infty$.} First,
we recall a slight variant of
\cite[Theorem 9]{benyi}:
\begin{theorem}\label{bnj}
Let $\Phi$ be a phase
function as in Proposition
\ref{proposizione}. Let
$\chi$ be a smooth function
satisfying $\chi(\eta)=1$ for
$|\eta|\leq R$,
$\chi(\eta)=0$ for
$|\eta|\geq2R$, for some
$R>0$. Then for every compact
subset $K\subset\rd$ there
exists a constant $C_{K}>0$
such that
\begin{equation}\label{fl1}
\sup_{x\in K} \|e^{2\pi i
\Phi(x,\cdot)}\chi\|_{\cF
L^1} < C_{K}.
\end{equation}
\end{theorem}
The proof is a
straightforward
generalization of \cite[Theorem 9]{benyi},
where the case of phases
independent of $x$ was
considered. Namely, since the
parameter $x$ varies in a
compact set $K$, all the
estimates given there hold
uniformly with respect to $x$
(more generally, Theorem
\ref{bnj} holds for phases
positively homogeneous of
order $\alpha>0$ with respect
to $\eta$, but here we are
only interested in the case
$\alpha=1$).
\par We also observe that, if
$\phi\in \cC_0^\infty(\rd)$, then
\begin{equation}\label{fiD}
    \|\phi(D) u\|_{L^\infty}\leq C \|u\|_{L^\infty},\quad \quad \forall u\in\cS(\rd),
\end{equation}
where $\phi(D)
u=\cF^{-1}(\phi\hat{u})=\hat{\phi}\ast
{u}$. This is a consequence of Young's
inequality, since $\hat{\phi}\in
L^1$.\par

We have now all the pieces in
place to prove the
boundedness on
$\mathcal{M}^\infty$ of the
FIO $T$. Since its symbol
$\sigma$ vanishes for
$|\o|\geq R$, taking $\chi$
as in Theorem \ref{bnj}, we
have $ \sigma \phas
=\sigma\phas\chi(\o) $ and,
for every $v\in \cS(\rd)$,
\begin{align*}
T v(x)&=\intrd\intrd e^{2\pi i (\Phi\phas -y\o)} \sigma\phas \chi(\o) v(y)\,d\o dy\\
&= \intrd\cF[(e^{2\pi i
(\Phi(x,\cdot)}\chi)(\sigma(x,\cdot))](y)
v(y)\, dy.
\end{align*}

If we set $K=\{x\in\R^d:\
|x|\leq R\}$, since $\cF L^1$
is an algebra under pointwise
multiplication and using the
majorization \eqref{fl1}, we
obtain
\begin{equation}\label{b1bis}
\|Tv\|_{L^\infty}\leq
\|v\|_{L^\infty}\sup_{K}(\|e^{2\pi
i \Phi(x,\cdot)}\chi\|_{\cF
L^1}\|\sigma(x,\cdot)\|_{\cF
L^1})\leq C\|v\|_{L^\infty}.
\end{equation}

Since $\sigma\phas=0$ for
$|\o|\geq R$, for every
$\phi\in \cS(\rd)$, with
$\phi(\o)\equiv 1$ for
$|\o|\leq R$, we have as well
$$
Tu= T(\phi(D)u),
$$
 so that, using the embeddings
 $L^\infty\hookrightarrow M^\infty$, and
  \eqref{b1bis} for $v=
  \phi(D)u$,
 $$
 \|Tu\|_{M^\infty}\leq  \|Tu\|_{L^\infty}=
 \|T(\phi(D)u)\|_{L^\infty}\leq C
 \|\phi(D)u\|_{L^\infty}.
  $$
 Now, choose a  function $\f$ as in \eqref{normmp}.
  Then $\f$ satisfies ${\rm supp}\,(T_m\f)\cap$ ${\rm
  supp}\,\phi\not=\emptyset$ for finitely many
  $m\in \zd$ only. Hence, the estimate \eqref{fiD} yields, for $u\in\cS(\rd)$,
\begin{align*}\|\phi(D)u\|_{L^\infty}&\lesssim
 \sum_{m\in\zd}\|\phi(D)\f(D-m) u\|_{L^\infty}\lesssim
 \sup_{m\in\zd}\|\phi(D)\f(D-m) u\|_{L^\infty}\\
&\lesssim
\sup_{m\in\zd}\|\f(D-m)
u\|_{L^\infty}\\&\asymp\|u\|_{M^\infty}.
\end{align*}
\noindent So the FIO  $T$ is bounded on
 $\mathcal{M}^\infty$.
 This concludes the proof of
 Proposition \ref{proposizione}.

\section{Oscillations at
infinity}\label{section5}

In this section we prove
Theorem \ref{maintheorem} for
an operator satisfying the
assumptions stated there and
and whose symbol $\sigma$
satisfies, in addition,
\[
\sigma(x,\eta)=0\quad {\rm
for}\ |\eta|\leq 2.
\] We first
perform a further
reduction.\par
 For every
$(x_0,\eta_0)\in \Lambda'$,
$|\eta_0|=1$, there exist an
open neighborhood
$\Omega\subset\R^d$ of $x_0$,
an open conic neighborhood
$\Gamma\subset\R^d\setminus\{0\}$
of $\eta_0$ and $\delta>0$
such that
\begin{equation}\label{i1}
|\det
\partial_{x,\eta}\Phi(x,\eta)|\geq\delta>0,\quad
\forall
(x,\eta)\in\Omega\times\Gamma,
\end{equation}
and
\begin{equation}\label{i2}
\forall x\in\Omega,\
\textrm{the map}\
\Gamma\ni\eta\mapsto
\nabla_x\Phi(x,\eta)\
\textrm{is a diffeomorphism
onto the range},
\end{equation}
\begin{equation}\label{i3}
\forall \eta\in\Gamma,\
\textrm{the map}\ \Omega\ni
x\mapsto
\nabla_\eta\Phi(x,\eta)\
\textrm{is a diffeomorphism
onto the range}.
\end{equation}
Hence, by a compactness
argument and a finite
partition of unity we can
assume that $\sigma$ itself
is supported in a cone of the
type $\Omega'\times\Gamma'$,
for some open
$\Omega'\subset\R^{d}$,
$\Gamma'\subset\R^{d}\setminus\{0\}$
conic,
$\Omega'\subset\subset\Omega$,
$\Gamma'\subset\subset\Gamma$,
with $\Phi$ satisfying the
above conditions on
$\Omega\times\Gamma$.\par
We
now prove the boundedness of
an operator $T$ of order
$m=-d/2$ on $M^1$ and on
$\mathcal{M}^\infty$. Since
it is a classical fact that
FIOs of order 0 are
continuous on
 $L^2=M^2$, (see e.g.
 \cite[page 402]{stein93}), the desired
 continuity result on $M^p$ when
 $m=-d|1/2-1/p|$,
 $1<p<\infty$,
  follows by complex
 interpolation.\par
 In detail, the interpolation step goes as follows.
 Observe first that, for
 every $s\in\R$, the operator
 $\langle D\rangle^s$ defines
 an isomorphism of $\cM^p_s$
 onto $\cM^p$. This follows easily from the
 characterization of the
  $M^p_s$ norm in
  \eqref{normmp}, after
  writing $\varphi=\tilde{\varphi}\varphi$,
  for some
  $\tilde{\varphi}\in\cC^\infty_0(\R^d)$, $\tilde{\varphi}\equiv1$
  on ${\rm supp}\,\varphi$,
combined with the fact that
   the multiplier
   $\langle\eta\rangle^s\tilde\varphi(\eta-m)\langle
   m\rangle^{-s}$
   is in $\Fur L^1$ uniformly with respect to
   $m$.\\
   Hence, the
 operator $T=T\langle D\rangle^{-s}\langle D\rangle^s$ is bounded $\cM^p_s\to \cM^p$
 if and only if
 $T\langle D\rangle^{-s}$ is
 bounded on $\cM^p$. Observe
 moreover that $T\langle
 D\rangle^{-s}$ is a FIO
 with the same phase as $T$, and symbol
 $\sigma(x,\eta)\langle\eta\rangle^{-s}$,
 which has order $m-s$.\\
 Suppose now that the
 desired result is already
 obtained for $p=1,2$.
 Take $1<p<2$ and
consider a FIO $T$ of order
 $m=-d(1/p-1/2)$. Then, taking the above remarks into
 account, $T$ extends to a
 bounded operator
 $M^1_{m+d/2}\to M^1$ and
 $M^2_m\to M^2$. Hence, the
  boundedness on $M^p$ follows
 by complex interpolation, i.e. \eqref{interpmp}, because, if
 $\theta\in(0,1)$ satisfies
 $(1-\theta)/1+\theta/2=1/p$,
 one has
 $(m+d/2)(1-\theta)+m\theta=0$. The
 proof for $2<p<\infty$ is
 similar.\par Of course, when
 in \eqref{soglia}
 there is a strict inequality, the desired result  follows from the
  equality-case, for an operator with order $m'<m$ has also order $m$. \par
Hence, from now on, we assume
$m=-d/2$ and prove the
boundedness of $T$ on $M^1$
and on $\mathcal{M}^\infty$.

\subsection{Boundedness on
$M^1$} We need the following result
(cf. \cite{grobner,sugimototomita}).
\begin{lemma}\label{le41}
Let $\chi$ be a smooth function
supported where $B_0^{-1}\leq|\eta|\leq
B_0$, for some $B_0>0$. Then, for every $u\in\cS(\rd)$,
\[
\sum_{j=1}^\infty\|\chi(2^{-j}
D)u\|_{M^1}\lesssim \|u\|_{M^1},
\]
where $\chi(2^{-j}
D)u= \cF^{-1}[\chi(2^{-j}\cdot)\hat{u}]$.
\end{lemma}
\begin{proof}
Let $u\in\cS(\rd)$.
By using the characterization
of the $M^1$-norm in \eqref{normmp},
we have
\begin{align}
\sum_{j=1}^\infty\|\chi(2^{-j}
D)u\|_{M^1}
&\asymp\sum_{j=1}^\infty\sum_{m\in\mathbb{Z}^d}
\|\f(D-m)\chi(2^{-j}D)u\|_{L^1}\\
&=\sum_{m\in\mathbb{Z}^d}\sum_{j=1}^\infty
\|\chi(2^{-j}D)\f(D-m)u\|_{L^1}\\
&\lesssim\sum_{m\in\mathbb{Z}^d}\sup_{j\geq1}
\|\chi(2^{-j}D)\f(D-m)u\|_{L^1}.
\end{align}
In the last inequality we
used the fact that, for any
$m$, the number of indices
$j\geq1$ for which ${\rm
supp}\
\chi(2^{-j}\,\cdot)\cap {\rm
supp}\
\f(\cdot-m)\not=\emptyset$
is uniformly bounded with
respect to $m$.\\
 The Fourier multiplier
 $\chi(2^{-j} D)$ is bounded
 on $L^1$, uniformly with
 respect to $j$. Indeed, for every $f\in\cS$,
 \begin{align*} \|\chi(2^{-j} D)f\|_{L^1}&=\|   \cF^{-1}(\chi(2^{-j} \cdot))\ast f\|_{L^1}\leq\|                \cF^{-1}(\chi(2^{-j} \cdot))\|_{L^1}\|f\|_{L^1}\\
 &=2^{jd} \|    \cF^{-1}(\chi)(2^{j} \cdot)\|_{L^1}\|f\|_{L^1}= \|              \cF^{-1}(\chi)( \cdot)\|_{L^1}\|f\|_{L^1}.
 \end{align*}
 Hence, $\|\chi(2^{-j}D)\f(D-m)u\|_{L^1}\lesssim \|\f(D-m)u\|_{L^1}$. Finally, using \eqref{normmp},
 \[\sum_{j=1}^\infty\|\chi(2^{-j}
D)u\|_{M^1}
\lesssim\sum_{m\in\mathbb{Z}^d}
\|\f(D-m)u\|_{L^1}\asymp
\|u\|_{M^1}.
\]
\end{proof}\\
 Consider now the usual
Littlewood-Paley
decomposition of the
frequency domain. Namely,
 fix a smooth function $\psi_0(\eta)$
  such that $\psi_0(\eta)=1$
  for $|\eta|\leq1$ and
  $\psi_0(\eta)=0$ for
  $|\eta|\geq2$. Set
  $\psi(\eta)=\psi_0(\eta)-\psi_0(2\eta)$,
  $\psi_j(\eta)=\psi(2^{-j}\eta)$, $j\geq1$.
  Then
  \[
  1=\sum_{j=0}^\infty\psi_j(\eta),\quad
\forall\eta\in\R^d.
 \]
  Notice
that, if $j\geq 1$, $\psi_j$ is supported
where $2^{j-1}\leq|\eta|\leq
2^{j+1}$. Since $\sigma\phas=0$, for $|\o|\geq 2$, we can write
\[
T=\sum_{j\geq1} T^{(j)},
\]
where $T^{(j)}$ has symbol
$\sigma_j(x,\eta):=\sigma(x,\eta)\psi_j(\eta)$.
Moreover, we observe that
\[
T^{(j)}=U_{2^{j/2}}\tilde{T}^{(j)}
U_{2^{-j/2}},
\]
where $\tilde{T}^{(j)}$ is
the FIO with phase
\begin{equation}\label{i100bis}
{\Phi}_j(x,\eta):=\Phi(2^{-j/2}x,2^{j/2}\eta)=2^{j/2}\Phi(2^{-j/2}x,\eta),
\end{equation}
and symbol
\[
\tilde{\sigma}_j(x,\eta):=
\sigma_j(2^{-j/2}x,2^{j/2}\eta),
\]
and $U_\lambda f(y)=f(\lambda
y)$, $\lambda>0$, is the
dilation operator. From Theorem \ref{dilprop} we have
\begin{equation}\label{di1}
\|U_{\lambda}f\|_{M^1}\lesssim
\|f\|_{M^1},\quad
\lambda\geq1,
\end{equation}
and
\begin{equation}\label{di2}
\|U_{\lambda}f\|_{M^1}\lesssim
\lambda^{-d}\|f\|_{M^1},\quad
0<\lambda\leq1.
\end{equation}
Assume for a moment that
\begin{equation}\label{intermedia}
\|\tilde{T}^{(j)}
u\|_{M^1}\lesssim
2^{-jd/2}\|u\|_{M^1}.
\end{equation}
Then, using \eqref{di1} and
\eqref{di2} we obtain
\[
\|T^{(j)} u\|_{M^1}\leq
2^{jd/2}2^{-jd/2}\|u\|_{M^1}=\|u\|_{M^1}.
\]
Actually, for the
frequency localization of
$T^{(j)}$,  the following  finer
estimate holds:
\[
\|T^{(j)}
u\|_{M^1}=\|T^{(j)}(\chi(2^{-j}D)
u)\|_{M^1}\leq\|\chi(2^{-j}D)u\|_{M^1},
\]
where $\chi$ is a smooth
function satisfying
$\chi(\eta)=1$ for
$1/2\leq |\eta|\leq2$ and
$\chi(\eta)=0$ for
$|\eta|\leq1/4$ and $|\eta|\geq4$ (so that
$\chi\psi=\psi$). Summing on
$j$ this last estimate with the aid of
Lemma \ref{le41} we obtain
\[
\|Tu\|_{M^1}\lesssim\|u\|_{M^1},
\]
which is the desired
estimate.\par It remains to
prove \eqref{intermedia}.
This follows from Proposition
\ref{pro2} applied to the
operator
$2^{jd/2}\tilde{T}^{(j)}$.
Indeed, it is easy to see
that the hypotheses are
satisfied uniformly with
respect to $j$. Precisely, we
observe that, for every $j\geq 1$,
\[
|\partial^\alpha_x\partial^\beta_\eta
\tilde{\sigma}_j(x,\eta)|\lesssim
2^{-j\frac{d}{2}-j\frac{|\alpha|+|\beta|}{2}},
\]
and
$\tilde{\sigma}_j(x,\eta)$ is
supported where
$2^{j/2-1}\leq|\eta|\leq
2^{j/2+1}$,
$x\in{\Omega}'_j:=\{2^{j/2}x,\
x\in\Omega'\}$,
$\eta\in\Gamma'$. Moreover,
after setting
${\Omega}_j:=\{2^{j/2}x,\
x\in\Omega\}$, we see that
\begin{equation}\label{i10bis}
|\partial^\alpha_x\partial^\beta_\eta
{\Phi}_j(x,\eta)|\lesssim
2^{j\left(1-\frac{|\alpha|+|\beta|}{2}\right)},
\end{equation}
for $(x,\eta)$ in the set
${\Omega}_j\times\Gamma,
2^{j/2-2}\leq|\eta|\leq
2^{j/2+2}$, which contains an
$\epsilon$-neighborhood of
${\rm
supp}\,\tilde{\sigma}_j$,
with $\epsilon$ independent
of $j$. Finally \eqref{i1},
\eqref{i2}, \eqref{i3} give
\begin{equation}\label{i1bis}
\left|{\rm det}\,
\left(\frac{\partial^2\Phi_j}{\partial
x_i\partial \eta_l}\Big|_{
(x,\eta)}\right)\right|\geq\delta>0,\quad
\forall
(x,\eta)\in{\Omega}_j\times\Gamma,
\end{equation}
and
\begin{equation}\label{i2bis}
\forall x\in{\Omega}_j,\
\textrm{the map}\
\Gamma\ni\eta\mapsto
\nabla_x{\Phi}_j(x,\eta)\
\textrm{is a diffeomorphism
onto the range},
\end{equation}
\begin{equation}\label{i3bis}
\forall \eta\in\Gamma,\
\textrm{the map}\
{\Omega}_j\ni x\mapsto
\nabla_\eta{\Phi}_j(x,\eta)\
\textrm{is a diffeomorphism
onto the range}.
\end{equation}
Hence Proposition \ref{pro2}
 applies and gives
 \eqref{intermedia}.
 \subsection{Boundedness on
 $\mathcal{M}^\infty$}
 We need the following result (cf. \cite{grobner,sugimototomita}).
 \begin{lemma}\label{lemma2}
For $k\geq0$, let $f_k\in\cS(\rd)$
satisfy ${\rm
supp}\,\hat{f}_0\subset
B_2(0)$ and
\[
{\rm
\supp}\,\hat{f}_k\subset\{\eta\in\R^d:\
2^{k-1}\leq|\eta|\leq
2^{k+1}\},\quad k\geq1.
\]
Then, if the sequence $f_k$ is bounded in $M^\infty§(\rd)$, the series
$\sum_{k=0}^\infty {f}_k$
converges in
${M}^\infty(\R^d)$ and
\begin{equation}\label{b0}
\|\sum_{k=0}^\infty
f_k\|_{M^\infty}\lesssim\sup_{k\geq0}\|f_k\|_{M^\infty}.
\end{equation}
 \end{lemma}
 \begin{proof}
The convergence of the series
$\sum_{k=0}^\infty {f}_k$ in
${M}^\infty(\R^d)$ is straightforward.\\
We then prove the desired estimate.
 Choose a window
function $g$ with ${\rm
supp}\,\hat{g}\subset
B_{1/2}(0)$. We can write
\[
V_g(f_k)(x,\omega)=(\hat{f}_k\ast
M_{-x}\hat{g})(\omega).
\]
Hence, ${\rm
\supp}\,V_g(f_0)\subset
B_{5/2}(0)\subset
B_{2^2}(0)$, and
\begin{align*}
{\rm
\supp}\,V_g({f}_k)&\subset\{\eta\in\R^d:\
2^{k-1}-2^{-1}\leq|\eta|\leq
2^{k+1}+2^{-1}\}\\
&\subset\{\eta\in\R^d:\
2^{k-2}\leq|\eta|\leq
2^{k+2}\},
\end{align*}
for $k\geq1$. Hence, for each
$(x,\omega)$, there are at most four
nonzero terms in the sum
$\sum_{k=0}^\infty V_g(f_k)(x,\omega)$.
Using this fact we obtain
\begin{align}
\|\sum_{k=0}^\infty
{f}_k\|_{M^\infty}&\asymp\|\sum_{k=0}^\infty
V_g({f}_k)\|_{L^\infty}\leq\|\sum_{k=0}^\infty
|V_g(f_k)|\|_{L^\infty}\nonumber\\
&\leq
4\|\sup_{k\geq0}|V_g(f_k)|\|_{L^\infty}=4\sup_{k\geq0}
\|V_g(f_k)\|_{L^\infty}\asymp\sup_{k\geq0}
\|f_k\|_{M^\infty}.\nonumber
\end{align}
\end{proof}\\
We now proceed in the proof
of the boundedness of $T$ (of
order $m=-d/2$) on
$\mathcal{M}^\infty$. The
Littlewood-Paley
decomposition of the
frequency domain, introduced
at the beginning of this
section, and Lemma
\ref{lemma2} yield
\begin{align}
\|Tu\|_{M^\infty}&=\|\sum_{k\geq0}\psi_k(D)Tu\|_{M^\infty}\nonumber\\
&\lesssim\sup_{k\geq0}\|\psi_k(D)Tu\|_{M^\infty}\nonumber\\
&\leq\sup_{k\geq0}\sum_{j=1}^\infty\|\psi_k(D)T^{(j)}
u\|_{M^\infty},\label{continuazione}
\end{align}
where, as before,
$T=\sum_{j=1}^\infty
T^{(j)}$, with $T^{(j)}$
having symbol
$\sigma_j(x,\eta):=\sigma(x,\eta)\psi_j(\eta)$
and the same phase $\Phi$.
Notice that the sequence of
symbols $\sigma_j(x,\eta)$ is
bounded in $S^{-d/2}_{1,0}$,
whereas the sequence of
symbols $\psi_k(\eta)$ is
bounded in $S^0_{1,0}$.
\par Applying  Theorem
\ref{composition} to each product
$\psi_k(D)T^{(j)}$,  we have
\[
\psi_k(D)T^{(j)}=S_{k,j}+R_{k,j},
\]
where $S_{k,j}$ are FIOs with the same
phase $\Phi$ and symbols $\sigma_{k,j}$
belonging to bounded subset of
$S^{-d/2}_{1,0}$, supported in
\begin{equation}\label{b1}
\{(x,\eta)\in\Omega\times\Gamma:\
|\nabla_x\Phi(x,\eta)|\leq2,
2^{j-1}\leq|\eta|\leq
2^{j+1}\},\quad {\rm if}\
k=0,
\end{equation}
and in
\begin{equation}\label{b2}
\{(x,\eta)\in\Omega\times\Gamma:\
2^{k-1}\leq|\nabla_x\Phi(x,\eta)|\leq2^{k+1},
2^{j-1}\leq|\eta|\leq
2^{j+1}\},\quad {\rm if}\
k\geq1.
\end{equation}
The operators  $R_{k,j}$ are
smoothing operators whose
symbols $r_{k,j}$ are
 in a bounded subset of $\mathcal{S}(\R^{2d})$,
 supported where
$2^{j-1}\leq|\eta|\leq 2^{j+1}$.\par
Observe that, by the Euler's identity and \eqref{i1},
\[
|\nabla_x\Phi(x,\eta)|=
|\partial^2_{x,\eta}\Phi(x,\eta)\eta|\asymp|\eta|,\quad
\forall (x,\eta)\in\Omega\times\Gamma.
\]
Inserting this equivalence in
\eqref{b1} and \eqref{b2}, we
obtain that there exists
$N_0>0$ such that
$\sigma_{k,j}$ vanishes
identically if $|j-k|>N_0$.
Whence, the right-hand side
in \eqref{continuazione} is
seen to be
\[
\leq\sup_{k\geq0}\sum_{j\geq1:
|j-k|\leq N_0}\|S_{k,j}u
\|_{M^\infty}+\sup_{k\geq0}
\sum_{j=1}^\infty
\|R_{k,j}u\|_{M^\infty}.\]
This expression will be
dominated by the $M^\infty$
norm of $u$ if we prove
that\footnote{Of course, as
always we mean that the
constant which is implicit in
the notation $\lesssim$ is
independent of $k,j$.}
\begin{equation}\label{parteprincipale}
\|S_{k,j}u\|_{M^\infty}\lesssim\|u\|_{M^\infty},
\end{equation}
and
\begin{equation}\label{resto}
\|R_{k,j}u\|_{M^\infty}\lesssim
2^{-j}\|u\|_{M^\infty}.
\end{equation}
This last estimate is easy to
obtain. Namely, observe that
the symbols
$2^{j}r_{k,j}(x,\eta)$ are
still in a bounded subset of
$\mathcal{S}(\R^{2d})$: since $|\o|\asymp 2^{j}$, on
the support of $r_{k,j}$, for every
$\alpha\in\Z_+$, $N\geq0$,
\[
|\partial^\alpha_x\partial^\beta_\eta
r_{k,j}(x,\eta)|\lesssim_{\alpha,\beta,N}
(1+|x|+|\eta|)^{-N}\lesssim
2^{-j}(1+|x|+|\eta|)^{-N+1}.
\]
 This implies that
the corresponding operators
$2^jR_{k,j}$ are uniformly bounded on
$\mathcal{M}^\infty$, i.e. \eqref{resto}.\par It
remains to prove
\eqref{parteprincipale}. To this end,
we make use of the dilation operator
$U_\lambda$, as before. Precisely, we
recall from Theorem \ref{dilprop}
that
\begin{equation}\label{di1bis}
\|U_{\lambda}f\|_{M^\infty}\lesssim
\|f\|_{M^\infty},\quad
\lambda\geq1,
\end{equation}
and
\begin{equation}\label{di2bis}
\|U_{\lambda}f\|_{M^\infty}\lesssim
\lambda^{-d}\|f\|_{M^\infty},\quad
0<\lambda\leq1.
\end{equation}
Write
\[
S_{k,j}:=U_{2^{j/2}}
\tilde{S}_{k,j} U_{2^{-j/2}},
\]
where $\tilde{S}_{k,j}$ is
the FIO with phase
${\Phi}_j(x,\eta)$ defined in
\eqref{i100bis}, and symbol
\[
\tilde{\sigma}_{k,j}(x,\eta):=
\sigma_{k,j}(2^{-j/2}x,2^{j/2}
\eta).
\]
Hence, taking into account
\eqref{di1bis}, \eqref{di2bis}, we see
that \eqref{parteprincipale} will
follow from
\[
\|\tilde{S}_{k,j}
u\|_{M^\infty}\lesssim
2^{-jd/2}\|u\|_{M^\infty}.
\]
This last estimate is a
consequence of Proposition
\ref{pro2} applied to
$2^{jd/2}\tilde{S}_{k,j}$.
Indeed, using the notation
above, namely
${\Omega}'_j:=\{2^{j/2}x,\
x\in\Omega'\}$,
${\Omega}_j:=\{2^{j/2}x,\
x\in\Omega\}$, we already
observed that \eqref{i10bis},
\eqref{i1bis}, \eqref{i2bis},
\eqref{i3bis} hold. Moreover,
$\tilde{\sigma}_{k,j}(x,\eta)$
is supported where
$2^{j/2-1}\leq|\eta|\leq
2^{j/2+1}$,
$x\in{\Omega}'_j$,
$\eta\in\Gamma'$, and
satisfies
\[
|\partial^\alpha_x\partial^\beta_\eta
\tilde{\sigma}_{k,j}(x,\eta)|\lesssim
2^{-j\frac{d}{2}-j\frac{|\alpha|+|\beta|}{2}}.
\]
This concludes the proof of Theorem \ref{maintheorem}.

\section{Sharpness of the results}\label{sharp}
In this section we prove the
sharpness of Theorems
\ref{maintheorem2} and
\ref{maintheorem}. Precisely,
for every $m>\displaystyle -d
\left|\frac12-\frac1p\right|$,
$1\leq p\leq\infty$, there
are FIOs of the type
\eqref{FIO}, satisfying the
assumptions in the
Introduction, which do not
extend to bounded operators
on $\Fur L^p_{\rm comp}$,
$1\leq p<\infty$, nor on the
closure of
$\mathcal{C}^\infty_0(\R^d)$
in $\Fur L^\infty_{\rm
comp}$, and therefore do not
extend to bounded operators
on $M^p$, $1\leq p<\infty$,
nor on
$\mathcal{M}^\infty$.\par The
key idea is that the
composition operator
$f\rightarrow f\circ \f$ with
$\f :\R\to\R $ being a
non-linear $C^1$ change of
variables, is unbounded on
the space $\cF L^p(\R)_{\rm
loc}$
\cite{beurling53,lebedev94}.\par
\subsection{Some auxiliary results}\ \\ We need to recall
the van der Corput Lemma
(see, e.g., \cite[Proposition
2, page 332]{stein93}).
 \begin{lemma}\label{vandercorput} Suppose $\phi$ is real-valued and smooth in $(a,b)\subset \R$, and that $|\phi^{(k)}(t)|\geq 1$, for all $t\in (a,b)$ and for some $k\geq 2$. Then, for every $\lambda>0$,
 \begin{equation}\label{vanderc}
 \left |\int_a^b e^{i\lambda \phi(t)}\,dt\right|\leq c_k \lambda^{-1/k},
 \end{equation}
 where the bound $c_k$ is independent of $\phi$ and $\lambda$.
 \end{lemma}

\begin{proposition}\label{prop1} Let $\f: \R\to \R$ be a
$\cC^{\infty}$
diffeomorphism, whose
restriction to the interval
$(0,1)$ is a non-linear
 diffeomorphism on $(0,1)$. This means that there
  exists an interval $I\subset (0,1)$ such that
  $|\f^{''}(t)|\geq \rho>0$, for all  $t\in I$.
  Let $\chi\in\cC^{\infty}_0(\R)$, $\chi\geq 0$,
   with supp $\chi\subset (0,1)$ and
    $|({\rm supp} \chi)\cap \f(I)|>0$. Then, if we set
\begin{equation}\label{fn}
f_n (t)= \chi(t) e^{2\pi i nt},\quad
n\in\N,
\end{equation}
for $1\leq p\leq 2$, we have
\begin{equation}\label{fnest}
\|f_n\circ\f\|_{\cF L^p} \geq
c(p,\f,\chi) \, n^{1/p-1/2},\quad
\quad\forall n\in\N.
\end{equation}
\end{proposition}
\begin{proof} It follows the pattern of
 \cite[page 219]{lebedev94}. In place of the
  sequence $\{e^{int}\}_{n\geq1}$ used there,
   here we consider the smooth compactly supported sequence of
 functions $\{f_n\}_{n\geq1}$ in \eqref{fn}.

The assumptions on $\chi$,
Parseval's identity, and
H\"older's inequality yield
\begin{align}
0<C&=\int_I \chi(\f(t))\,dt=\left|\int_{\R} f_n(\f(t)) (\mathds{1}_{I}(t)e^{-2\pi i n\f(t)})\,dt\right|\nonumber\\
&=\left|\int_{\R} \cF^{-1}(f_n\circ\f)(\eta)\cF(\mathds{1}_{I}e^{-2\pi i n\f})(\eta)\,d\eta\right|\nonumber\\
&\leq
\|\cF(f_n\circ\f)\|_{L^p}\|\cF(\mathds{1}_{I}e^{-2\pi
i n\f})\|_{L^{p'}}=\|f_n\circ\f\|_{\cF
L^p}\|\mathds{1}_{I}e^{-2\pi i
n\f}\|_{\cF L^{p'}}, \label{st1}
\end{align}
for $1/p+1/p'=1$.

 Let us estimate $\|\mathds{1}_{I}e^{-2\pi i n\f}
 \|_{\cF L^{p'}}$. The van der Corput Lemma (Lemma
 \ref{vandercorput}),
 for $\lambda= 2\pi\rho n$, $\phi(t)=\phi_{n,\eta}(t)
 =-(\f(t)/\rho+t\eta/(\rho n))$, hence
$|\phi^{''}(t)|=\displaystyle\left|
\frac{\f^{''}(t)}{\rho}\right|\geq
1$, $[a,b]=\overline{I}$,
gives
$$\left|\cF( \mathds{1}_{I}e^{-2\pi i n\f})(\eta)\right|
=\left|\int_a^b e^{2\pi i
\rho n\phi(t)}\,dt\right|\leq
c_2 (2\pi \rho n)^{-1/2},
\quad \forall
n\in\N,\,\,\forall \eta\in\R,
$$
with the constant $c_2$
independent of $n,\eta$.
Taking the $L^\infty$-norm,
\begin{equation}\label{car1}
\|\mathds{1}_{I}e^{-2\pi
i n\f}\|_{\cF L^{\infty}}\leq
c(\rho) n^{-1/2}.
\end{equation}
On the other hand,
\begin{equation}\label{car2}
\|\mathds{1}_{I}e^{-2\pi
i n\f}\|_{\cF
L^{2}}=\|\mathds{1}_{I}e^{-2\pi
i n\f}\|_{ L^{2}}=|I|^{1/2}.
\end{equation}
For $1<p<2$, (hence
$2<p'<\infty$), H\"older's
inequality gives
\begin{align*}
\int_{\R}\left|\cF( \mathds{1}_{I}e^{-2\pi i n\f})(\eta)\right|^{p'}\,d\eta &\leq\|\left|\cF( \mathds{1}_{I}e^{-2\pi i n\f})\right|^{p'-2}\|_{L^\infty}\|\left|\cF( \mathds{1}_{I}e^{-2\pi i n\f})\right|^{2}\|_{L^1}\\
&=\| \mathds{1}_{I}e^{-2\pi i
n\f}\|^{p'-2}_{\cF L^\infty} \|
\mathds{1}_{I}e^{-2\pi i
n\f}\|^2_{L^2},
\end{align*}
whence
$$\| \mathds{1}_{I}e^{-2\pi i n\f}\|_{\cF L^{p'}}\leq
\| \mathds{1}_{I}e^{-2\pi i n\f}\|^{1-2/p'}_{
\cF L^{\infty}}\| \mathds{1}_{I}e^{-2\pi i n\f}\|^{2/p'
}_{ L^{2}}\leq c(p,\f) n^{-1/2+1/p'}.
$$
This last estimate holds for
$p=1,2$ too (because of
\eqref{car1} and
\eqref{car2}), and inserted
in \eqref{st1} gives
$$\|f_n\circ\f\|_{\cF L^p}\geq c(p,\f,\chi)\, n^{1/p-1/2},
$$
for $1\leq p\leq 2$, as
desired.
\end{proof}

The generalization to  dimension $d\geq
1$ reads as follows.
\begin{corollary} Let $\f$ be as in Proposition \ref{pro1}
and $f_n$ defined in
\eqref{fn}. We define
\begin{equation}\label{fntilde}
\tilde{f}_n(t_1,\dots,t_d)=f_n(t_1)\cdots
f_n(t_d),\quad
\tilde{\f}(t_1,\dots,t_d)=(\f(t_1),\dots,\f(t_d)),
\end{equation}
then
\begin{equation}\label{estfn}
\| \tilde{f}_n\circ \tilde{\f}\|_{\cF
L^p(\rd)} \geq c(p,\f,\chi)
\,n^{d(1/p-1/2)},
\end{equation}
for $1\leq p\leq 2$.
\end{corollary}
We also need the following
result.
\begin{lemma}
Let $h\in\cS(\rd)$, $\la
t\ra^m=(1+|t|^2)^{m/2}$, $m\in\R$.
Then, for $y\in\rd$ and $1\leq
p\leq\infty$,
\begin{equation}\label{trans}
    \| h T_y\la\cdot\ra^m\|_{M^p} \leq c(h,p) \la y\ra^m.
\end{equation}
\end{lemma}
\begin{proof} For a non-zero window function $g\in\cS(\rd)$, we have
$$V_g( hT_y\la\cdot\ra^m)\phas =\intrd
 e^{-2\pi i \o t} \la t-y\ra^m  h(t)  \overline{g(t-x)}\,dt
$$
Let us show that the STFT
$V_g (hT_y\la\cdot\ra^m)$ is
in $L^p(\rdd)$ with the
majorization \eqref{trans}.
For $N_1\in \N$, an
integration by parts gives
$$ V_g( hT_y\la\cdot\ra^m)\phas =(1+|2\pi\o|^2)^{-N_1}
\intrd e^{-2\pi i \o t}
(1-\Delta_t)^{N_1}[ \la
t-y\ra^m
 h(t)  \overline{g(t-x)}] dt.
$$
By Petree's inequality,
$$|\partial^\a_t \la t-y\ra^m|\lesssim
\la t-y\ra^{m-|\a|}\lesssim
\la y\ra^{m-|\a|} \la
t\ra^{|m-|\a||} ,\quad\forall
\a\in\zd_+,\,\, m\in\R.$$ The
functions $g,h$ are in $\cS(\rd)$,
so that
$$|\partial^\b_t g(t-x)|\lesssim \la t-x\ra^{-N_2}
\lesssim \la x\ra^{-N_2}\la
t\ra^{N_2},\,\,
|\partial^\gamma_t
h(t)|\lesssim \la
t\ra^{-N_3},\,\,\forall
N_2,N_3\in\N, \,\,\forall
\beta,\gamma\in\zd_+.
$$

Hence
\begin{align*}|V_g( hT_y\la\cdot\ra^m)\phas|&\lesssim
(1+|\o|^2)^{-N_1}\la
x\ra^{-N_2} \sup_{|\a|\leq
N_1}
 \la y\ra^{m-|\a|}\la t\ra^{|m-|\a||+N_2-N_3}\\
&\lesssim  \la y\ra^{m} \la
t\ra^{|m|+N_1+N_2-N_3}\la
\o\ra^{-2N_1}\la
x\ra^{-N_2},\quad \forall
N_1,N_2,N_3\in\N.
\end{align*}
Choosing $N_1, N_2$ such that
$2pN_1>d$, $pN_2>2$, and
$N_3$ such that
$p(N_3-N_1-N_2-|m|)>d$, we
attain the desired result.
\end{proof}

We use the previous lemma to compute the action of the multiplier
$\la D\ra^m$ on the functions $\tilde{f}_n$. Precisely,
\begin{corollary}
Let $m\in\R$ and $\tilde{f}_n$ defined
in \eqref{fntilde}. Then,
\begin{equation}\label{dfnf}
    \|\la D\ra^{m} \tilde{f}_n\|_{M^p}\leq c(\chi,p) n^m.
\end{equation}
\begin{proof} Using the invariance of the modulation spaces $M^p$ under \ft, we have
$$\|\la D\ra^{m} \tilde{f}_n\|_{M^p}\asymp \|\la \cdot \ra^{m} \widehat{\tilde{f}_n}\|_{M^p}=\|\la \cdot \ra^{m} T_{\tilde{n}} \widehat{\tilde{\chi}}\|_{M^p}=\|(T_{-\tilde{n}}\la \cdot \ra^{m})  \widehat{\tilde{\chi}}\|_{M^p},
$$
with
$\tilde{\chi}(t_1,\dots,t_d)=({\chi}\otimes\cdots\otimes{\chi})(t_1,\dots,t_d)$,
$\chi$ defined in Proposition \ref{prop1}, and
$\tilde{n}=(n,\dots,n)$. The previous lemma and the estimate $\la
\tilde{n}\ra \lesssim d^{1/2} n$  yield the majorization
\eqref{dfnf}.
\end{proof}
\end{corollary}
\medskip
We can now prove the
sharpness of Theorems
\ref{maintheorem2} and
\ref{maintheorem}. It is
clear that it would be
sufficient to prove the
sharpness of Theorem
\ref{maintheorem2}, because
of Lemma \ref{lloc} and the
fact that our operators have
symbols compactly supported
in $x$. However, we start
with showing the optimality
of Theorem \ref{maintheorem}
first, and then we show how
the argument above in fact
gives the optimality of
Theorem \ref{maintheorem2} as
well.\par
\subsection{Sharpness of Theorem
\ref{maintheorem}}\ \\
{\bf Sharpness for} $1\leq
p\leq2$. Consider the FIO
$$T_{\tilde{\f}} f(x)= f\circ \tilde{\f}(x)=
\intrd e^{2\pi i
\tilde{\f}(x) \eta}
\hat{f}(\eta)\,d\eta,$$ where
$\tilde{\f}$ is defined in
\eqref{fntilde}. We require
that
the one-dimensional
diffeomorphism $\f$ satisfies
the assumptions of
Proposition \ref{prop1} and
the additional hypothesis
\begin{equation}\label{bound}0<c\leq |\f'(x)|\leq C,
\quad \forall x \in \R,
\end{equation}
Then, the phase
$\Phi(x,\eta)= \tilde{\f}(x)
\eta$ fulfills the standard
assumptions in the
Introduction; in particular
it is non-degenerate. Notice
that $T_{\tilde{\f}}$ maps
$\mathcal{C}^\infty_0(\R^d)$
into itself and ${\rm
supp}\,T_{\tilde{\f}}f\subset
(0,1)^d$ if ${\rm
supp}\,{f}\subset (0,1)^d$.
\par We are interested in a
FIO with symbol $\sigma$ of
order $m$ and with compact
support with respect to the
$x$-variable. So, let
$G\in\cC^{\infty}_0(\rd)$,
 $G\geq 0$ and $G\equiv1
$ on $[0,1]^d$, and consider
the FIO $F$ defined by
\begin{equation}\label{fioF}
Ff(x)=G(x)[(T_{\tilde{\f}} \la D\ra^m
)f](x)=\intrd e^{2\pi i \tilde{\f}(x)
\eta} G(x)\la\eta\ra^m
\hat{f}(\o)\,d\o.
\end{equation}
The symbol
$\sigma(x,\o)=G(x)\la\eta\ra^m
$ is of order $m$ with
compact support in $x$. So,
if $m$ satisfies
\eqref{soglia}, Theorem
\ref{maintheorem} assures the
boundedness of  $F$ on $M^p$.
We now show that this
threshold is sharp for $1\leq
p\leq2$. Indeed, consider the
functions $\tilde{f}_n$ in
\eqref{fntilde}. They are
supported in $(0,1)^d$, so
$T_{\tilde{\f}}\tilde{f}_n$
are. Hence, applying the
estimate \eqref{estfn} and
Lemma \ref{lloc}, we obtain
\begin{align*}
n^{d(1/p-1/2)}&\lesssim \|
\tilde{f}_n\circ \tilde{\f}\|_{\cF
L^p(\rd)}=\|
T_{\tilde{\f}}\tilde{f}_n\|_{\cF
L^p(\rd)}
=\| GT_{\tilde{\f}}\tilde{f}_n\|_{\cF L^p(\rd)}\\&\asymp \| GT_{\tilde{\f}}\tilde{f}_n\|_{M^p(\rd)}=\| GT_{\tilde{\f}}\la D\ra^m\la D\ra^{-m}\tilde{f}_n\|_{M^p(\rd)}\\
&\lesssim \|F\|_{M^p\rightarrow M^p}
\|\la
D\ra^{-m}\tilde{f}_n\|_{M^p(\rd)}\lesssim
\|F\|_{M^p\rightarrow M^p} \,n^{-m},
\end{align*}
where  the last inequality is
due to \eqref{dfnf}. For
$n\to\infty$, we obtain
$-m\geq d(1/p-1/2)$, i.e.,
\eqref{soglia}. \par

\textbf{Sharpness for
$2<p\leq\infty$.}
  Observe that the adjoint operator
$T^*_{\tilde{\f}}$ of the
above FIO $T_{\tilde{\f}}$ is
still a FIO given by
$$T^*_{\tilde{\f}} f(x)=\frac1{|J_{\tilde{\f}}
(\tilde{\f}^{-1}(x))|}\intrd
e^{2\pi i
\widetilde{\f}^{-1}(x)
\eta}f(\o)\,d\o,$$ with
$\widetilde{\f}^{-1}(x_1,\dots,x_d)=
(\f^{-1}(x_1),\dots,
\f^{-1}(x_d))$ and
$|J_{\tilde{\f}}|$ the
Jacobian of $\tilde{\f}$. Its
phase $\Phi\phas
=\widetilde{\f}^{-1}(x) \eta$
still fulfills the standard
assumptions.\par
 Now, let
$H\in \cC^{\infty}_0(\rd)$ ,
$H\geq 0$, and  $H(x)\equiv
1$ on supp ($G\circ
\f^{-1}$). We define the
operator
\begin{equation}\label{ftilde}
\tilde{F}f(x)=H(x)[\la
D\ra^m T^*_{\tilde{\f}} (G
f)](x).
\end{equation}
Using Theorem
\ref{composition}, it is
easily seen that $\tilde{F}$
is a FIO of order $m$, with
symbol compactly supported in
the $x$-variable. Its adjoint
is given by
\begin{equation}\label{ftildestar}
\tilde{F}^*= G
T_{\tilde{\f}}\la D\ra^m H=
F+ R,
\end{equation}
where $F$ is defined in
\eqref{fioF} and the
remainder $R$ is given by
\[
R f(x)=G(x)
[T_{\tilde{\f}}\la D\ra^m
((H-1)f)](x).\] If we choose
a function
$\tilde{G}\in\cC^{\infty}_0(\rd)$
, $\tilde{G}\equiv 1$ on supp
$G$ we can write
\begin{align*} Rf&=
\tilde{G}(x)G(x)
[T_{\tilde{\f}}\la D\ra^m
((H-1)f)](x)\\&=\tilde{G}(x)T_{\tilde{\f}}[(G\circ
\tilde{\f}^{-1})\la D\ra^m
((H-1)f)](x).
\end{align*}
 By assumptions,  supp $(G\circ \f^{-1})\cap$
supp $(H-1)=\emptyset$, so
that the pseudodifferential
operator
$$f\longmapsto (G\circ \f^{-1})\la D\ra^m ((H-1)f)$$
is a regularizing operator
 (it immediately follows by the composition
  formula of pseudodifferential operators, see e.g.
   \cite[Theorem 18.1.8, Vol. III]{hormander}):
 this means that it maps $\cS'(\rd)$ into $\cS(\rd)$. The operator  $T_{\tilde{\f}}$ is a
 smooth change of variables, so $\tilde{G}(x)T_{\tilde{\f}}$   maps
$\cS(\rd)$  into itself. To
sum up, the remainder
operator $R$ maps $\cS'(\rd)$
into $\cS(\rd)$, hence it is
bounded on $M^{p}$. This
means that $\tilde{F}^*$ is
continuous on some $M^{p}$
iff $F$ is.

The operator   $\tilde{F}$ is
a FIO of the type
\eqref{FIO}, with symbol of
order $m$ and compactly
supported in the $x$
variable. Hence it is bounded
on $M^{p}$ if $m$ fulfills
\eqref{soglia}. We now show
that this threshold is sharp
for $2<p<\infty$. Indeed, if
 $\tilde{F}$ were
bounded on $M^{p}$, then its
adjoint $\tilde{F}^*$ would
be bounded on
$(M^{p})'=M^{p'}$, with
$1<p'<2$, and
 the same for $F$. But the former case gives
 the boundedness of
$F$ on  $M^{p'}$ iff $-m\geq d(1/p'-1/2)=d(1/2-1/p)$,
 that is the
desired threshold. For
$p=\infty$, if $\tilde{F}$
were bounded on
$\mathcal{M}^\infty$, its
adjoint  $\tilde{F}^*$ would
be bounded on
$(\mathcal{M}^\infty)'=M^{1}$
and the former argument
applies.
\subsection{Sharpness of Theorem
\ref{maintheorem2}}\ \\
We start with an elementary
remark. Consider a FIO $T$
satisfying the hypotheses in
the Introduction.
 Suppose
that it does {\it not}
satisfy an estimate of the
type
\[
\|Tu\|_{M^p}\leq
C\|u\|_{M^p},\quad \forall
u\in\mathcal{S}(\R^d),
\]
(hence
$m>-d\left|\frac{1}{2}-\frac{1}{p}\right|$).
Suppose, in addition, that
the distribution kernel of
$K(x,y)$ of $T$ has the
property that the two
projections of ${\rm supp}\,
K$ on $\R^d_x$ and $\R^d_y$
are bounded sets. Then, by
Lemma \ref{lloc} one sees
that that there exist compact
subsets $K,K'\subset\R^d$ and
a sequence of Schwartz
functions $u_n$,
$n\in\mathbb{N}$, such that
\[
{\rm supp}\, u_n\subset
K,\quad {\rm supp}\,
Tu_n\subset K',\quad \forall
n\in\mathbb{N},
\]
and
\[
\|Tu_n\|_{\Fur L^p}\geq
n\|u_n\|_{\Fur L^p},\quad
\forall n\in\mathbb{N}.
\]
Hence $T$ does not extend to
a bounded operator on  $\Fur
L^p_{\rm comp}$, if $1\leq
p<\infty$, nor on the closure
of the test functions in
$\Fur L^\infty_{\rm comp}$,
if $p=\infty$.\par Taking
this fact into account, we
see that the operator
$\tilde{F}$ in \eqref{ftilde}
provides the desired
counterexample for
$2<p\leq\infty$, if
$m>-d\left|\frac{1}{2}-\frac{1}{p}\right|$.\par
Similarly, the operator
$\tilde{F}^\ast$ in
\eqref{ftildestar} provides
the counterexample for $1\leq
p\leq2$.

\section*{Acknowledgements}
The authors would like to thank the
anonymous referee for helpful comments.

\end{document}